\documentclass{amsart}
\usepackage{verbatim,amssymb,amsmath,amscd,latexsym,amsbsy,mathrsfs} 
\usepackage{graphicx}
\input{xy}
\xyoption{all}

\newtheorem{thm}{Theorem}[section]

\newtheorem{rmk}[thm]{Remark}

\newtheorem{con}[thm]{Conjecture}

\DeclareMathOperator*{\im}{im}

\newcommand{\PP}{\mathbb{P}}

\newcommand {\C} {{\mathbb C}}
\newcommand {\R} {{\mathbb R}}
\newcommand {\Z} {{\mathbb Z}}
\newcommand {\Q} {{\mathbb Q}}

\newcommand {\E} {{\mathcal E}}

 \newtheorem{lemma}{Lemma}[section]
 \newtheorem{prop}{Proposition}[section]

\begin{document}
\title{ Geometric  Hodge structures with prescribed Hodge numbers.}
\author{
        Donu Arapura    
}
\thanks {Partially supported by the NSF and at the IAS by the Ralph E. and
  Doris M. Hansmann fund.}

\address{Department of Mathematics\\
Purdue University\\
West Lafayette, IN 47907\\
U.S.A.}
 \maketitle

Let us call a pure, and necessarily effective  polarizable,  Hodge
structure geometric if it is contained in  the cohomology of a smooth complex  projective
variety. A natural question is  which polarizable  Hodge structures are
geometric?  Our goals in this note are twofold.
We first of all show that there are no numerical
restrictions on geometric Hodge structures. That is for any set of Hodge numbers (subject to
the obvious constraints imposed by Hodge symmetry and effectivity), there exists a geometric
Hodge structure with precisely these Hodge numbers.  The argument is
fairly easy, but we hope that there is some pedagogical value in
writing it out. Our example is contained in
the cohomology of a power $E^N$ of a CM  elliptic curve. Since the
Hodge conjecture is known for such varieties, we get a  slightly
stronger statement, that our example is the Hodge realization of
a Grothendieck  motive.  The second part of this note
is less elementary and  more speculative. We try to address the question about where  all the  two dimensional
irreducible geometric  Hodge structures come from. The answer, assuming the
Hodge conjecture, is that they come from either from elliptic curves or motives over $\bar
\Q$. Examples of the latter type include the Hodge structures constructed
in the first section, and Hodge structures  associated to
modular forms discussed in the second. Although we do not have a good
conjecture about  how to describe all the two dimensional $\bar \Q$-motivic Hodge
structures, the ones arising from motives over $\Q$  should be exactly
the modular examples.

Coming back to part one,
we should mention that Schreieder \cite{scr} has solved the related,
but more difficult, problem of finding a smooth projective variety with
a prescribed set of Hodge numbers in a given degree $k$ (under
suitable hypotheses when $k$ is even). This does
imply the main result here for odd $k$. Nevertheless, the construction
is different and somewhat more involved.
We thank Greg Pearlstein for first bringing  Schreieder's work to our
attention, and the referee for  other useful remarks.

\section{Main theorem}

We will work exclusively with rational Hodge structures below.
It is worth observing that since  the category of polarizable Hodge
structures is semisimple \cite{moonen}, we have:

\begin{prop}\label{prop:1}
A geometric Hodge structure  of weight $k$  is a direct summand of
$H^k(X,\Q)$ for some   smooth complex  projective
variety $X$.
\end{prop}

\begin{thm}
  Given a set of nonnegative integers $g^{k,0}, g^{k-1,1},\ldots
  g^{0,k}$, satisfying $k\ge 0$ and $g^{p,q}=g^{q,p}$, there exists a geometric
  rational  Hodge
  structure $H$ with $\dim H^{p,q}=g^{p,q}$.
\end{thm}

\begin{lemma}\label{lemma:1}
The Tate structure $\Q(-n)$ is geometric for $n\ge 0$.
  Let $V_1$ and $V_2$ be geometric Hodge structures of weight $n_1$
  and $n_2$ respectively, then $V_1\otimes V_2$ is geometric. If $n_1=n_2$,
  then $V_1\oplus V_2$ is also geometric.
\end{lemma}

\begin{proof}
We have $\Q(-n) =H^{2n}(\PP^n,\Q)$.
  By assumption, there exists smooth projective varieties $X_i$ such
  that $V_i$ is a summand of $H^{n_i}(X_i)$. Then $V_1\otimes V_2$ and
  $V_1\oplus V_2$ are summands of $H^k(X_1\times X_2)$ where
  $k=n_1n_2$ and $k=n_1=n_2$ respectively.
\end{proof}

\begin{lemma}\label{lemma:2}
  For any $n>0$, there exists a geometric Hodge structure $G$ with
  $\dim G^{n,0}= \dim G^{0,n}=1$ and the remaining Hodge numbers equal
  to zero.
\end{lemma}

\begin{proof}
Although, it is not difficult to prove by hand, it is a bit cleaner if we
make use of basic facts about Mumford-Tate groups, cf \cite[appendix
B]{lewis} or \cite{moonen}.
In order to spell things out as explicitly as possible,  we work with
the specific  elliptic curve
$E=\C/\Z i\oplus \Z$. This has complex multiplication by
$K=\Q(i)$. Let $V= H^1(E, \Q)$. We choose a basis $v_1, v_2$ for $V$ dual to
the basis $i, 1$ of the lattice $\Z i\oplus \Z$ . Then $dz= iv_1+ v_2$
and $d\bar z= -iv_1+v_2$ determines the Hodge decomposition on
$V\otimes \C$. We define  a homomorphism $h$ of the unit circle into
$GL(V\otimes \R)$ by rotations
$$h(\theta) =
\begin{pmatrix}
  \cos\theta & -\sin \theta\\ \sin\theta & \cos\theta
\end{pmatrix}
$$
Then $h(\theta)dz =e^{i\theta}dz$, and $h(\theta)d\bar
z=e^{-i\theta}d\bar z$. Let $SMT(V)$ denote the Weil restriction
$$Res_{K/\Q}\mathbb{G}_m = \left\{
\begin{pmatrix}
  a & -b \\ b & a
\end{pmatrix}
\mid (a,b)\in \Q^2, (a,b)\not= (0,0) \right\}
$$
This is clearly  the smallest $\Q$-algebraic group whose real points contain
$\im h$. In other words, $SMT(V)$ is the special Mumford-Tate group or
Hodge group of $V$. The
Mumford-Tate group $MT(V) $ is defined similarly, as the smallest
$\Q$-algebraic group whose real points contain the image of Deligne's
torus $\mathbb{S}(\R)=\C^*$. It  works out to the product of $SMT(V)$ and the group
of nonzero scalar matrices.  The significance of this group comes from
the Tannakian interpretation: $MT(V)$ is the group whose category of
representations is equivalent to the tensor category generated by
$V$. Since $MT(V)$ was described as a matrix group, it comes with an
obvious representation $\rho:MT(V)\to  GL(V)$. For an integer $n$, let
$V_n$ be the representation of this group given by composing $\rho$ with
the $n$th power homomorphism $MT(V)\to MT(V)$. Now let us suppose that
$n>0$. Then $V_n$ is irreducible, and therefore simple as a  Hodge
structure. The elements $dz, d\bar z$ still give a basis of $V_n
\otimes \C$, but
now the circle acts by $dz\mapsto e^{in\theta} dz$ and $d\bar z\mapsto
e^{-in\theta} dz$. Thus these vectors span $V_n^{n,0}$ and
$V_n^{0,n}$.  We have an embedding $V_n\subseteq
S^nV\subset V^{\otimes n}$ given by identifying it with the span of $v_1^n$ and
$v_2^n$. Therefore the previous lemma implies that $V_n$ is geometric.
Thus $G=V_n$ is  the desired Hodge structure.
\end{proof}

\begin{lemma}\label{lemma:3}
  Given integers $p>q\ge 0$, there exists a
  geometric Hodge structure $H(p,q)$ with $\dim H(p,q)^{p,q}=\dim
  H(p,q)^{q,p}=1$ and the remaining Hodge numbers equal to zero.
\end{lemma}

\begin{proof}
Let $G$ be the Hodge structure
  constructed in the previous lemma with $n=p-q$. Then $H(p,q)= G\otimes \Q(-q)$ will satisfy the above conditions,.
\end{proof}

\begin{rmk}
  We can replace $E$ by any CM elliptic curve $E'$, and the same construction
  works. We denote the corresponding Hodge structure by  $H_{E'}(p,q)$. However, the method will fail for non CM curves, because the
  Mumford-Tate group will no longer be abelian.
\end{rmk}

\begin{proof}[Proof of theorem]
We get the desired Hodge structure by taking  sums of  the  Hodge structures
constructed in lemma \ref{lemma:3}, and  the appropriate number of Tate structures when $k$ is even.
More explicitly
$$H= \bigoplus_{p>q} H(p,q)^{\oplus g^{pq}} \> \underbrace{\oplus
  \Q(-k/2)^{\oplus g^{k/2,k/2}}}_{\text{$k$ even}}$$
\end{proof}

\begin{rmk}
  As the referee points out, this can also proved by invoking
  \cite{abdulali} but we prefer to keep the argument elementary and
  self-contained.
\end{rmk}

The proof actually shows that $H\subset H^k(E^N,\Q)$ for some $N$. We
can use this fact to get the stronger conclusion stated in the
introduction; it says roughly that the inclusion is also defined algebro-geometrically.
In order to make a precise statement, we recall that a motive, or more
precisely an effective  pure motive with respect to homological equivalence,
consists of a smooth projective variety $X$ together with an
algebraic cycle $p\in H^*(X\times X,\Q)$ such that $p\circ p=p$
\cite{kleiman}. The composition $\circ$ is the usual one for
correspondences \cite[chap 16]{fulton}. It follows that $p$ is an idempotent of
the ring of Hodge endomorphisms $\prod_i End_{HS}(H^i(X))$. Thus we get a mixed Hodge
structure $p(H^*(X,\Q))$ which is by definition the Hodge realization of the
motive $(X,p)$. We quickly run up against fundamental difficulties. For
instance, it is unknown,  whether $H^k(X)$
is the realization of a motive for arbitrary $X$ and $k$. For this, we would need to
know that the K\"unneth components of the diagonal $\Delta\subset
X\times X$ are algebraic, and this is one Grothendieck's standard
conjectures \cite{groth}. Fortunately,
it is not an issue in our example, because the Hodge conjecture holds
for $E^{2N}$ \cite[appendix B, \S 3]{lewis} and this implies
Grothendieck's conjecture. Thus $H^k(E^N)$ is the
realization of a motive. Since $H$ is a summand of $H^k(E^N)$, it is
given by the image of an idempotent in $End_{HS}(H^k(E^N))$. By the
Hodge conjecture, this is algebraic. Thus to summarize:

\begin{prop}
 The  geometric Hodge structure $H$, given in the theorem, is the Hodge realization of a
  motive.
\end{prop}

\section{Two dimensional geometric Hodge structures}

Our goal now is to obtain an understanding of where  two
dimensional geometric Hodge structures come from.  We may as well
assume that the structures are irreducible because otherwise they are
just sums of Tate structures. Given a subfield
$K\subset \C$, by a motive over $K$  we will mean a pair $(X,p)$, as
described earlier, such that both
$X$ and the algebraic cycle representing $p$ are defined over $K$.

\begin{thm}
Let $H$ be an irreducible two dimensional geometric Hodge structure.
Then, assuming the Hodge conjecture, there are two nonexclusive
possibilities.
\begin{enumerate}
\item[(a)] $H=H^1(E)(-m)$, where $E$ is an elliptic curve and $m\ge
  0$, or
\item[(b)] $H$ comes from a motive over $\bar \Q$.
\end{enumerate}
\end{thm}

\begin{proof}
  We recall that the level of $H$ is the largest value of
$|p-q|$ such that $H^{pq}\not=0$.  If the level is $1$ then
after Tate twisting, we can assume that $H$ is of type
$(1,0),(0,1)$ and  of course still $2$ dimensional. So it  must coincide with
$H^1(E)$ where $E= H/(H^{10}+H_\Z)$ for some lattice $H_\Z\subset H_\Q$.
 If  the level of $H$ is greater than $1$, we obtain (b) from the
next proposition.
\end{proof}

\begin{prop}\label{prop:1}
  Assuming the Hodge conjecture, any  geometric Hodge structure with
  no consecutive nonzero Hodge numbers comes from a motive over $\bar
  \Q$. In particular, this conclusion applies to two dimensional geometric Hodge
  structures of level at least two.
\end{prop}

\begin{proof}
Fix a period domain $D$ parametrizing polarized Hodge
structures with fixed  Hodge numbers $h^{m,0}=*,
h^{m-1,1}=0,h^{m-2,2}=*,\ldots$ \cite{cmp}.
Pick a point  $t_0\in D$ corresponding to a Hodge structure $H$ which lies in $
H^m(X,\Q)$, where $X$ is a complex smooth projective variety. Since
the category of polarizable Hodge structures is semisimple,  $H$ is the image of an
idempotent $p\in End_{HS}(\bigoplus_i H^i(X,\Q))$. By the Hodge conjecture, this can be lifted to an
algebraic cycle that we also denote by $p=\sum n_ip_i$ on $X\times X$.
The pair $(X,p)$ is defined  over a finitely generated extension of
$K/\bar \Q$. We can regard $K$ as the function field of a  variety
$S/\bar \Q$. After shrinking $S$ if necessary,  we can assume it is smooth and that
$X$ is the generic fibre (based changed to $\C$)
of a smooth projective family $f:\mathcal{X}\to S$. We can also
assume, after further shrinking, that the relative 
cycles  $\mathcal{P}_i\subset \mathcal{X}\times_S\mathcal{X}$
given by the scheme theoretic closure of  the components  $p_i$ are flat
over $S$. Let $\mathcal{P}$ denote the relative correspondence $ \sum n_i\mathcal{P}_i$. Then the
rank of  the image of the
action of the fibre $\mathcal{P}_s$ on $H^*(\mathcal{X}_s(\C))$ is constant as
$s$ varies. It follows that 
$im[\mathcal{P}:R^*f_*\Q\to R^*f_*\Q]$
defines a variation of Hodge structure over $S(\C)$.  Let $\pi:\tilde
S(\C)\to D$ be the associated period map on the universal cover. Now comes the key
point. With our choice of Hodge numbers, Griffiths transversality
forces $\pi$ to be constant. Therefore we can realize $t_0$ by the
motive determined by any pair $(\mathcal{X}_s,\mathcal{P}_s)$ with $s\in
S(\C)$. In particular, we can choose  $s\in S(\bar \Q)$.
\end{proof}

We describe a few examples of Hodge structures arising from motives
over $\bar \Q$. Suppose that $E$ is a
CM elliptic curve. Then as is well known \cite[chap II, \S 2]{silverman}, $E$ is defined over $\bar \Q$. Thus the examples 
$H_E(p,q)$ constructed earlier all arise from motives over $\bar \Q$.
 These examples have CM, and so
are rather special. Recall that a Hodge structure $H$ has CM if
$SMT(H)$ is abelian. It turns out that any irreducible two dimensional
polarizable Hodge structure of even weight has CM. To see this,
observe that the special Mumford-Tate group would be abelian
since it would have to  contain $SO_2$ as a maximal compact. (Our  thanks
to the referee for this remark.) It is not difficult to conclude, as a
consequence,  that two dimensional irreducible Hodge structures of even weight are
of the form $H_E(p,q)$, for some CM elliptic curve $E$.

In addition to the $H_E(p,q)$, there is one very natural source of examples: modular forms.  Let us
explain the basic set up. For
simplicity, we work with the 
principal congruence group $\Gamma(n)=\ker [SL_2(\Z)\to
SL_2(\Z/n)]$ with $n\ge 3$, but there are several other natural choices. Recall that a weight $k$ cusp form   $f\in
S_k(\Gamma(n))$  is given by a
holomorphic function on the upper half plane $\mathbb{H}$ satisfying 
$$f\left(\frac{a z+b}{cz+d}\right) = (cz+d)^{k}f(z),\quad
\begin{pmatrix}
  a & b\\ c& d
\end{pmatrix}\in \Gamma(n)
$$
and such that  the Fourier expansion $f =\sum a_j q^j$, $q=e^{2\pi iz/n}$, has
only positive terms \cite{ds}. 
The  moduli space $Y(n) = \mathbb{H}/\Gamma(n)$ of elliptic curves with  a full level
$n$ structure is a fine moduli space, so it comes with a universal family
$\pi_n:\E(n)\to Y(n)$. This can be completed to a minimal elliptic surface $
\bar \E(n)\to X(n)$ over the nonsingular compactification $j:Y(n)\hookrightarrow X(n)$. By a
theorem of Zucker \cite{zucker}, the intersection cohomology of the
symmetric power
$$H=IH^1(X(n),
S^{k-1} R^1\pi_{n*}\Q) = H^1(X(n), j_*S^{k-1}R^1 \pi_{n*}\Q)$$
 carries a pure Hodge
structure of weight $k$. This Hodge structure turns out to be
isomorphic to one constructed by Shimura \cite[\S 12]{zucker}. It has only $(k,0)$ and $(0, k)$ parts, and the $(k,0)$
part is isomorphic to the space $S_{k+1}(\Gamma(n))$ of cusp forms of weight $k+1$.
We note that $H$ is geometric, since it can be shown to
lie  the
cohomology of a desingularization of the $(k-1)$-fold fibre product
$\bar \E(n)\times_{X(n)}\times\ldots \bar \E(n)$. The quickest way to
see this is by applying the decomposition theorem for Hodge modules
\cite{saito}, although we won't actually need this. The space $X(n)$ has a
large collection of commuting self-correspondences called Hecke operators
\cite{ds}. These operators act on $H$ \cite{scholl}, and they can be used to 
 break it up into pieces.
 Suppose that $f= \sum a_j q^j\in S_{k+1}(\Gamma(n))$ is a suitably normalized nonzero simultaneous  eigenvector  for the Hecke operators,
and further assume that the coefficients $a_j$ are all rational.
Then $f$ and $\bar f$ span a two dimensional Hodge
structure $H(f)\subset H$. These examples arise from
motives over $\Q$  constructed by Scholl \cite{scholl}. We want to
argue that the converse is also true.

Given a smooth complex projective variety  $X$, we  can always choose
a model $X_k$ over a  finitely generated field $k$. Then in addition to 
singular cohomology $H^*(X,\Q)$ with its Hodge structure,  we can consider \'etale cohomology
$H_{et}^*(X_{\bar k},\Q_\ell)\cong H^*(X,\Q)\otimes \Q_\ell$ with its
canonical $Gal(\bar k/k)$-action \cite{milne}.  Given a modular form
$f$ of the above type, we can associate a Galois
representation by replacing ordinary cohomology by  \'etale cohomlogy
in the discussion of the previous paragraph (c.f. \cite{deligne}). 
Serre \cite{serre}  conjectured, and Khare and Wintenberger \cite{kw}
proved, that a certain general class of two dimensional
Galois representations agrees ``modulo $\ell$'' with those coming from
modular forms. Although we won't try to make this precise,  we will
spell out  one  key  consequence \cite[thm
6]{serre}:

\begin{thm}[Serre, Khare-Wintenberger]
  Given a smooth projective variety defined over $\Q$ such that
$H^m(X(\C))$ is two dimensional and of type $(m,0), (0,m)$, then the Galois
representation on $H_{et}^m(X_{\bar \Q},\Q_\ell)$ comes from a modular form of
weight $m+1$. 
\end{thm}

If we assume Tate's conjectures \cite{tate}, then this correspondence would
have to come from an isomorphism of the underlying motives. In particular, it should be compatible with
  the Hodge structures. From here it seems a short  step to
conjecture the following: 

\begin{con}\label{con:1}
  Any two  dimensional  Hodge
structure of odd weight associated to a motive over $\Q$ should  be isomorphic to one
given by a modular form.
\end{con}

This is  hardly an  original idea; a version of this is stated in the
introduction to \cite{fm}, where it is referred to as a ``well known conjecture''.
Having stated this  conjecture, it seems a good idea to actually try to
check it in some cases. A  natural source of two dimensional
  Hodge structures, beyond those already considered, comes
  from the world of Calabi-Yau (or CY)  threefolds. These are three
dimensional smooth projective varieties $X$ with trivial canonical
bundle and hence $h^{30}=1$. Usually one also requires
$h^{10}=h^{20}=0$ but  this plays no role here.  We say that $X$ is rigid if it has no infinitesimal deformations,
i.e.  $H^1(X,T)=0$. By Serre duality, this
is equivalent to the vanishing of $h^{12}=h^{21}$. Thus $H^3(X,\C) =
H^{30}\oplus H^{03}$ is two dimensional. There is by now a large collection
of known examples of rigid  CY's defined over $\Q$,  see \cite{meyer}. In  this note  we will be content to consider
just  one example.  We recall that to  a fan, by which we mean
subdivision of a Euclidean space  into rational polyhedral cones, we can associate a
special kind of algebraic variety called a toric variety \cite{fulton1}. This
construction can be applied  to a root lattice, where the Weyl chambers form a fan. 
Verrill \cite{verrill} constructed a rigid
CY $3$-fold $V$ by first starting with the toric variety associated to the
root lattice $A_3$,   taking a double cover branched over an appropriate divisor in twice
the  anti-canonical system, and then choosing a suitable desingularization.
It can be described  more directly as a
desingularization of  the hypersurface 
$$ (x+y+z+w)(yzw+zxw+xyw+xyz)t - (t + 1)^2 · xyzw = 0$$
  in $\PP^3\times \PP^1$ where  $t$ an inhomogeneous coordinate on the
  second factor. Using this equation, Saito and Yui  \cite[thm 5.1]{sy} gives an explicit birational equivalence
  between Verrill's variety $V$ and a desingularization $Y$  of the  fibre product  of the elliptic modular surface
  associated to the congruence group $\Gamma_1(6)\subset SL_2(\Z)$
  with itself. Since $h^{30}$ is a birational invariant,
  $h^{30}(Y)=1$. So under this birational equivalence
  $H^3(V)$ will correspond to the necessarily unique irreducible sub Hodge structure of
  $H^3(Y)$ of type $(3,0), (0,3)$. On the other hand $IH^1(X_1(6),
  S^2R^1\pi_*\Q)\subseteq H^3(Y)$ is also of this type. So these must
  match, and the conjecture is verified in this example.

\end{document}